\documentclass[12pt]{article}
\usepackage[utf8]{inputenc}
\usepackage{graphicx}
\usepackage{amsmath}
\usepackage{fullpage}
\usepackage{mathtools}
\usepackage{csquotes}
\usepackage{amsfonts}
\usepackage{amssymb}
\usepackage{amsthm}
\newtheorem{thm}{Theorem}[section]
\newtheorem{lem}{Lemma}[section]

\newtheorem{ex}{Example}[section]

\newtheorem{cor}{Corollary}[section]

\title{Squared distance matrices of  trees with matrix weights}
\author{Iswar Mahato\thanks{Department of Mathematics, Indian Institute of Technology Kharagpur, Kharagpur 721302, India. Email: iswarmahato02@gmail.com} \and M. Rajesh Kannan\thanks{Department of Mathematics, Indian Institute of Technology Kharagpur, Kharagpur 721302, India. Email: rajeshkannan@maths.iitkgp.ac.in, rajeshkannan1.m@gmail.com }}
\date{\today}
\begin{document}
\maketitle
\baselineskip=0.25in

\begin{abstract}
	Let $T$ be a tree on $n$ vertices whose edge weights are positive definite matrices of order $s$. The squared distance matrix of $T$, denoted by $\Delta$, is the $ns \times ns$ block matrix with $\Delta_{ij}=d(i,j)^2$, where $d(i,j)$ is the sum of the weights of the edges in the unique $(i,j)$-path. In this article, we obtain a formula for the determinant of $\Delta$ and find ${\Delta}^{-1}$ under some conditions.
\end{abstract}

{\bf AMS Subject Classification(2010):} 05C22, 05C50.

\textbf{Keywords. } Tree, Distance matrix, Squared distance matrix, Matrix weight, Determinant, Inverse.

\section{Introduction}
Let $T$ be a tree with vertex set $V(T)=\{1,\hdots,n\}$ and edge set $E(T)=\{e_1,\hdots,e_{n-1}\}$. If two vertices $i$ and $j$ are adjacent, we write $i\sim j$. Let us assign an orientation to each edge of $T$. Two edges  $e_i=(p,q)$ and $e_j=(r,s)$ of $T$ are \textit{ similarly oriented} if $d(p,r)=d(q,s)$ and is denoted by $e_i\Rightarrow e_j$, otherwise they are   \textit{oppositely oriented} and is denoted by $e_i \rightleftharpoons e_j$. The \textit{edge orientation matrix} $H=(h_{ij})$ of $T$ is the $(n-1)\times (n-1)$ matrix whose rows and columns are indexed by the edges of $T$ and the entries are defined \cite{bapat2013product} as
$$h_{ij}=
\begin{cases}
	\text{$1$} & \quad\text{if $e_i\Rightarrow e_j$, $i \neq j$};\\
	\text{$-1$} & \quad\text{if  $e_i \rightleftharpoons e_j$, $i \neq j$};\\
	\text{$1$} & \quad\text{if $i=j$.}
\end{cases}$$
The \textit{incidence matrix} $Q$ of $T$ is the $n \times n-1$ matrix with its rows indexed by $V(T)$ and the columns indexed by $E(T)$. The entry corresponding to the row $i$ and column $e_j$ of $Q$ is $1$ if $e_j$ originates at $i$, $-1$ if $e_j$ terminates at $i$, and zero if $e_j$ and $i$ are not incident. We assume that the same orientation is used while defining the edge orientation matrix $H$ and the incidence matrix $Q$. 

The \emph{distance} between the vertices $i,j\in V(T)$, denoted by $d(i,j)$, is the length of the shortest path between them in $T$. The \emph{distance matrix} of $T$, denoted by $D(T)$, is the $n \times n$ matrix whose rows and columns are indexed by the vertices of $T$ and the entries are defined as follows:  $D(T)=(d_{ij})$, where $d_{ij}=d(i,j)$. In \cite{bapat2013product}, the authors introduced the notion of \emph{squared distance matrix} $\Delta$, which is defined to be the Hadamard product $D\circ D$, that is, the $(i,j)$-th element of $\Delta$ is $d_{ij}^2$. For the unweighted tree $T$, the determinant of $\Delta$ is obtained in \cite{bapat2013product}, while the inverse and the inertia of $\Delta$ are considered in \cite{bapat2016squared}. In \cite{bapat2019}, the author considered an extension of these results to a weighted tree whose each edge is assigned a positive scalar weight and found the determinant and inverse of $\Delta$. Recently, in \cite{das2020squared}, the authors determined the inertia and energy of the squared distance matrix of a complete multipartite graph. Also, they characterized the graphs among all complete $t$-partite graphs on $n$ vertices for which the spectral radius of the squared distance matrix and the squared distance energy are maximum and minimum, respectively.

In this article, we consider a weighted tree $T$ on $n$ vertices with each of its edge weights are positive definite matrices of order $s$. For $i,j \in V(T)$, the distance $d(i,j)$ between $i$ and $j$ is the sum of the weight matrices in the unique $(i,j)$-path of $T$. Thus, the distance matrix $D=(d_{ij})$ of $T$ is the block matrix of order $ns\times ns$ with its $(i,j)$-th block $d_{ij}=d(i,j)$ if $i\neq j$, and is the $s \times s$ zero matrix if $i=j$. The squared distance matrix $\Delta$ of $T$ is the $ns\times ns$ block matrix with its $(i,j)$-th block is equal to $d(i,j)^2$ if $i\neq j$, and is the $s \times s$ zero matrix if $i=j$. The Laplacian matrix $L=(l_{ij})$ of $T$ is the $ns \times ns$ block matrix defined as follows: For $i,j \in V(T)$, $i\neq j$, the $(i,j)$-th block $l_{ij}=-(W(i,j))^{-1}$ if $i \sim j$, where $W(i,j)$ is the matrix weight of the edge joining the vertices $i$ and $j$, and the zero matrix otherwise. For $i \in V(T)$, the $(i,i)$-th block of $L$ is $\sum_{j\sim i}(W(i,j))^{-1}$. 

In the context of classical distance, the matrix weights have been studied in \cite{atik2017distance} and \cite{Bapat2006}. The Laplacian matrix with matrix weights have been studied in \cite{atik2017distance,Sumit2022laplacian} and \cite{hansen2021expansion}. The Resistance distance matrix and the Product distance matrix with matrix weights have been considered in \cite{Atik-resistance}, and \cite{Product-matrix}, respectively. In this article, we consider the squared distance matrix $\Delta$ of a tree $T$ with matrix weights and find the formulae for the determinant and inverse of $\Delta$, which generalizes the results of \cite{bapat2013product,bapat2016squared,bapat2019}.

This article is organized as follows. In Section $2$, we define needed notations and state some preliminary results, which will be used in the subsequent sections. In Section $3$, we find some relations of Incidence matrix, Laplacian matrix, and Distance matrix with squared distance matrix. In Section $4$ and Section $5$, we obtain the formula for the determinant and inverse of $\Delta$, respectively.

\section{Notations and preliminary results}
In this section, we define some useful notations and state some known results which will be needed to prove our main results.

The $n\times 1$ column vector with all ones and the identity matrix of order $n$ are denoted by $\textbf{1}_n$ and $I_n$, respectively. Let $J$ denote the matrix of appropriate size with all entries equal to $1$. The transpose of a matrix $A$ is denoted by $A^{\prime}$. Let $A$ be an $n\times n$ matrix partitioned as
$ A=\left[ {\begin{array}{cc}
		A_{11} & A_{12} \\
		A_{21} & A_{22} \\
\end{array} } \right]$,
where $A_{11}$ and $A_{22}$ are square matrices. If $A_{11}$ is nonsingular, then the \textit{Schur complement }of $A_{11}$ in $A$ is defined as $A_{22}-A_{21}{A_{11}^{-1}}A_{12}$. The following is the well known Schur complement formula: $ \det A= (\det A_{11})\det(A_{22}-A_{21}{A_{11}^{-1}}A_{12})$. The\textit{ Kronecker product }of two matrices $A=(a_{ij})_{m\times n}$ and $B=(b_{ij})_{p\times q}$, denoted by $A\otimes B$, is defined to be the $mp\times nq$ block matrix $[a_{ij}B]$. It is known that for the matrices $A,B,C$ and $D$, $(A\otimes B)(C\otimes D)=AC\otimes BD$, whenever the products $AC$ and $BD$ are defined. Also $(A\otimes B)^{-1}=A^{-1}\otimes B^{-1}$, if $A$ and $B$ are nonsingular. Moreover, if $A$ and $B$ are $n \times n$ and $p\times p$ matrices, then $\det(A\otimes B)=(\det A)^p(\det B)^n$. For more details about the Kronecker product, we refer to \cite{matrix-analysis}.

Let $H$ be the edge-orientation matrix, and $Q$ be the incidence matrix of the underlying unweighted tree with an orientation assigned to each edge. The edge-orientation matrix of a weighted tree whose edge weights are positive definite matrices of order $s$ is defined by replacing $1$ and $-1$ by $I_s$ and $-I_s$, respectively. The incidence matrix of a weighted tree is defined in a similar way. That is, for the matrix weighted tree $T$, the edge-orientation matrix and the incidence matrix are defined as $(H\otimes I_s)$ and $(Q\otimes I_s)$, respectively.

Now we introduce some more notations. Let $T$ be a tree with vertex set $V(T)=\{1,\hdots,n\}$ and edge set $E(T)=\{e_1,\hdots,e_{n-1}\}$. Let $W_i$ be the edge weight matrix associated with each edge $e_i$ of $T$, $i=1,2,\hdots,n$. Let $\delta_i$ be the degree of the vertex $i$ and set $\tau_i=2-\delta_i$ for $i=1,2,\hdots,n$. Let $\tau$ be the $n \times 1$ matrix with components $\tau_1,\hdots,\tau_n$ and $\Tilde{\tau}$ be the diagonal matrix with diagonal entries $\tau_1,\tau_2,\hdots,\tau_n$. Let $\hat{\delta_i}$ be the matrix weighted degree of $i$, which is defined as 
$$\hat{\delta_i}=\sum_{j:j\sim i}W(i,j), ~~i= 1,\hdots,n.$$
Let $\hat{\delta}$ be the $ns\times s$ block matrix with the components $\hat{\delta_1},\hdots,\hat{\delta_n}$. Let $F$ be a diagonal matrix with diagonal entries $W_1,W_2,\hdots,W_{n-1}$. It can be verified that $L=(Q\otimes I_s){F}^{-1} (Q^{\prime}\otimes I_s)$. 

A tree $T$ is said to be directed tree, if the edges of the tree $T$ are directed.  If the tree $T$ has no vertex of degree $2$, then $\hat{\tau}$ denote the diagonal matrix with diagonal elements $1/\tau_1,1/\tau_2,\hdots,1/\tau_n$. In the following theorem, we state a basic result about the edge-orientation matrix $H$ of an unweighted tree $T$, which is a combination of Theorem $9$ of \cite{bapat2013product} and Theorem $11$ of \cite{bapat2016squared}. 

\begin{thm}\cite{bapat2013product,bapat2016squared}\label{detH}
	Let $T$ be a directed tree on $n$ vertices and let $H$ and $Q$ be the edge-orientation matrix and incidence matrix of $T$, respectively. Then $\det H=2^{n-2}\prod_{i=1}^n \tau_i$. Furthermore, if $T$ has no vertex of degree $2$, then $H$ is nonsingular and $H^{-1}=\frac{1}{2}Q^{\prime}\hat{\tau}Q$.
\end{thm}

Next, we state a known result related to the distance matrix of a tree with matrix weights.

\begin{thm}[{\cite[Theorem 3.4]{atik2017distance}}]\label{thm:DL}
	Let $T$ be a tree on $n$ vertices whose each edge is assigned a positive definite matrix of order $s$. Let $L$ and $D$ be the Laplacian matrix and distance matrix of $T$, respectively. If $D$ is invertible, then the following assertions hold:
	\begin{enumerate}
		\item $LD=\tau \textbf{1}_n^{\prime}\otimes I_s-2I_n\otimes I_s$.
		\item $DL=\textbf{1}_n{\tau}^{\prime}\otimes I_s-2I_n\otimes I_s.$
	\end{enumerate}
\end{thm}

\section{Properties of the squared distance matrices of trees }
In this section, we find the relation of the squared distance matrix with other matrices, such as distance matrix, Laplacian matrix, incidence matrix, etc. We will use these results to obtain the formulae for determinants and inverses of the squared distance matrices of directed trees.

\begin{lem}\label{lem:Ddel}
	Let $T$ be a tree with vertex set $\{1,2,\hdots,n\}$ and each edge of $T$ is assigned a positive definite matrix weight of order $s$. Let $D$ and $\Delta$ be the distance matrix and the squared distance matrix of $T$, respectively. Then
	$\Delta (\tau \otimes I_s) =D \hat{\delta}.$ 
\end{lem}
\begin{proof}
	Let $i \in \{1,2,\hdots,n\}$ be fixed. For $j \neq i$, let $p(j)$ be the predecessor of $j$ on the $(i,j)$-path of the underlying tree. Let $e_j$ be the edge between the vertices $p(j)$ and $j$. For $1 \leq j\leq n-1 $, let $W_j$ denote the weight of the edge $e_j$  and $X_j=\hat{\delta_j}-W_j$. Therefore, 
	\begin{eqnarray*}
		2\sum_{j=1}^n d(i,j)^2 &=&  \sum_{j=1}^n d(i,j)^2+\sum_{j\neq i} \Big(d(i,p(j))+W_j\Big)^2\\
		&=&\sum_{j=1}^n d(i,j)^2+\sum_{j\neq i} d(i,p(j))^2+2\sum_{j\neq i} d(i,p(j))W_j+\sum_{j\neq i}W_j^2.
	\end{eqnarray*}
	Since the vertex $j$ is the predecessor of $\delta_j-1$ vertices in the paths from $i$, we have
	$$\sum_{j\neq i} d(i,p(j))^2=\sum_{j=1}^n(\delta_j-1)d(i,j)^2.$$
	Thus,
	\begin{eqnarray*}
		2\sum_{j=1}^n d(i,j)^2 &=& \sum_{j=1}^n d(i,j)^2+\sum_{j=1}^n(\delta_j-1)d(i,j)^2+2\sum_{j\neq i} d(i,p(j))W_j+\sum_{j\neq i}W_j^2\\
		&=& \sum_{j=1}^n\delta_jd(i,j)^2+2\sum_{j\neq i} d(i,p(j))W_j+\sum_{j\neq i}W_j^2.
	\end{eqnarray*}
	Therefore, the $(i,j)$-th element of $\Delta (\tau \otimes I_s)$ is 
	\begin{align*}
		(\Delta (\tau \otimes I_s))_{ij}= \sum_{j=1}^n(2-\delta_j) d(i,j)^2=2\sum_{j\neq i} d(i,p(j))W_j+\sum_{j\neq i}W_j^2.   
	\end{align*}
	Now, let us compute the  $(i,j)$-th element of $D \hat{\delta}$.
	\begin{eqnarray*}
		(D \hat{\delta})_{ij}=\sum_{j=1}^n d(i,j)\hat{\delta_j} &=&  \sum_{j\neq i}\Big(d(i,p(j))+W_j\Big)(W_j+X_j)\\
		&=&\sum_{j\neq i} d(i,p(j))W_j+\sum_{j\neq i}W_j^2+\sum_{j\neq i}\Big(d(i,p(j))+W_j\Big)X_j.
	\end{eqnarray*}
	Note that $X_j$ is the sum of the weights of all edges incident to $j$, except $e_j$. Hence, 
	\begin{align*}
		\big(d(i,p(j))+W_j\big)X_j =d(i,j)X_j= \sum_{l\sim j,l\neq p(j)} d(i,p(l))W_l.
	\end{align*}
	Therefore, 
	$$\sum_{j\neq i}\big(d(i,p(j))+W_j\big)X_j=\sum_{j\neq i}\sum_{l\sim j,l\neq p(j)} d(i,p(l))W_l=\sum_{j\neq i} d(i,p(j))W_j. $$
	Thus,
	\begin{align*}
		(D \hat{\delta})_{ij}= \sum_{j=1}^n d(i,j)\hat{\delta_j}=2\sum_{j\neq i} d(i,p(j))W_j+\sum_{j\neq i}W_j^2=(\Delta (\tau \otimes I_s))_{ij}.  
	\end{align*}
	This completes the proof.
\end{proof}

\begin{lem}\label{lem:FHF}
	Let $T$ be a directed tree with vertex set $\{1,\hdots,n\}$ and edge set $\{e_1,\hdots,e_{n-1}\}$ with each edge $e_i$ is assigned a positive definite matrix weight $W_i$ of order $s$ for $1 \leq i \leq n-1$. Let  $H$ and $Q$ be the  edge orientation matrix and incidence matrix of $T$, respectively. 
	If $F$ is the diagonal matrix with diagonal entries $W_1,W_2,\hdots,W_{n-1}$, then 
	$$(Q^{\prime}\otimes I_s)\Delta (Q\otimes I_s)=-2F(H\otimes I_s)F.$$
\end{lem}
\begin{proof}
	For $i,j\in \{1,2,\hdots,n-1\}$, let $e_i$ and $e_j$ be two edges of $T$ such that $e_i$ is directed from $p$ to $q$ and $e_j$ is directed from $r$ to $s$. Let $W_i$ and $W_j$ be the  weights of the edges $e_i$ and $e_j$, respectively. If $d(q,r)=Y$, then it is easy to see that
	\begin{eqnarray*}
		\Big((Q^{\prime}\otimes I_s)\Delta (Q\otimes I_s)\Big)_{ij} &=&
		\begin{cases}
			\text{$(W_i+Y)^2+(W_j+Y)^2-(W_i+W_j+Y)^2-Y^2$,} & \text{if $e_i\Rightarrow e_j$,}\\
			\text{$-(W_i+Y)^2-(W_j+Y)^2+(W_i+W_j+Y)^2+Y^2$,}& \text{if $e_i \rightleftharpoons e_j$.}\\
		\end{cases}\\
		&=& 
		\begin{cases}
			\text{$-2W_iW_j$,} & \text{if $e_i\Rightarrow e_j$,}\\
			\text{$2W_iW_j$,}& \text{if $e_i \rightleftharpoons e_j$.}\\
		\end{cases}
	\end{eqnarray*}
	Note that $(F(H\otimes I_s)F)_{ij}=
	\begin{cases}
		\text{$W_iW_j$} & \quad\text{if $e_i\Rightarrow e_j$,}\\
		\text{$-W_iW_j$}& \quad\text{if $e_i \rightleftharpoons e_j$.}
	\end{cases}$\\
	Thus, $\Big((Q^{\prime}\otimes I_s)\Delta (Q\otimes I_s)\Big)_{ij}=-2(F(H\otimes I_s)F)_{ij}.$
\end{proof}

\begin{lem}\label{deltaL}
	Let $T$ be a  tree with vertex set $\{1,\hdots,n\}$ and edge set $\{e_1,\hdots,e_{n-1}\}$ with each edge $e_i$ is assigned a positive definite matrix weight $W_i$ of order $s$ for $1 \leq i \leq n-1$. Let $L,D$ and $\Delta$ be the Laplacian matrix, the distance matrix and the squared distance matrix of $T$, respectively. Then
	$\Delta L=2D(\Tilde{\tau}\otimes I_s)-\textbf{1}_n\otimes {\hat{\delta}^\prime}$.
\end{lem}
\begin{proof}
	Let $i,j\in V(T)$ and the degree of the vertex $j$ is $t$. Suppose $j$ is adjacent to the vertices $v_1,v_2,\hdots,v_t$, and let $e_1,e_2,\hdots,e_t$ be the corresponding edges with edge weights $W_1,W_2,\hdots,W_t$, respectively.\\
	\textbf{Case 1.} For $i=j$, we have 
	\begin{eqnarray*}
		(\Delta L)_{ii}&=&\sum_{s=1}^n d(i,s)^2 l_{si}\\
		&=&\sum_{s\sim i} d(i,s)^2 l_{si}\\
		&=& W_1^2(-W_1)^{-1}+\hdots +W_t^2(-W_t)^{-1}\\
		&=&-(W_1+W_2+\hdots +W_t)\\
		&=&-\hat{\delta_i}\\
		&=& \big(2D(\Tilde{\tau}\otimes I_s)-\textbf{1}_n\otimes {\hat{\delta}^\prime}\big)_{ii}.
	\end{eqnarray*}
	\textbf{Case 2.} Let $i\neq j$. Without loss of generality, assume that the $(i,j)$-path passes through the vertex $v_1$ (it is possible that $i=v_1$). If $d(i,j)=Z$, then $d(i,v_1)=Z-W_1$, $d(i,v_2)=Z+W_2$, $d(i,v_3)=Z+W_3$, $\hdots, d(i,v_t)=Z+W_t$. Therefore,
	\begin{eqnarray*}
		(\Delta L)_{ij}&=&\sum_{s=1}^n d(i,s)^2 l_{sj}\\
		&=&\sum_{s\sim j} d(i,s)^2 l_{sj}+d(i,j)^2 l_{jj}\\
		&=& {d(i,v_1)}^2(-W_1)^{-1}+{d(i,v_2)}^2(-W_2)^{-1}+\hdots +{d(i,v_t)}^2(-W_t)^{-1}+d(i,j)^2 l_{jj}\\
		&=&(Z-W_1)^2(-W_1)^{-1}+(Z+W_2)^2(-W_2)^{-1}+(Z+W_3)^2(-W_3)^{-1}\\
		& &+\hdots +(Z+W_t)^2(-W_t)^{-1}+Z^2\big((W_1)^{-1}+(W_2)^{-1}+\hdots+(W_t)^{-1}\big)\\
		&=&(W_1^2-2ZW_1)(-W_1)^{-1}+(W_2^2+2ZW_2)(-W_2)^{-1}+(W_3^2+2ZW_3)(-W_3)^{-1}\\ & & +\hdots+(W_t^2+2ZW_t)(-W_t)^{-1}\\
		&=&-(W_1+W_2+\hdots +W_t)+2Z-2(t-1)Z\\
		&=& 2(2-t)Z-(W_1+W_2+\hdots +W_t)\\
		&=& 2\tau_j Z-\hat{\delta_j}\\
		&=& \big(2D(\Tilde{\tau}\otimes I_s)-\textbf{1}_n\otimes {\hat{\delta}^\prime}\big)_{ij}.
	\end{eqnarray*}
	This completes the proof.
\end{proof} 

\section{Determinant of the squared distance matrix}
In this section, we obtain a formula for the determinant of the squared distance matrix of a tree with positive definite matrix weights. First, we consider the trees with no vertex of degree $2$.

\begin{thm}\label{det1}
	Let $T$ be a tree on $n$ vertices, and let  $W_i$ be the weights of the edge $e_i$, where $W_i$'s  are positive definite matrices of order $s$, $i=1,2,\hdots,n-1$. If $T$ has no vertex of degree $2$, then 
	$$\det (\Delta)=(-1)^{(n-1)s}2^{(2n-5)s}\prod_{i=1}^n {(\tau_i)^s}\prod_{i=1}^{n-1}\det (W_i^2) \det\bigg(\sum_{i=1}^n \frac{\hat{\delta_i}^2}{\tau_i}\bigg ).$$
\end{thm}
\begin{proof}
	Let us assign an orientation to each edge of $T$, and let $H$ be the edge orientation matrix and $Q$ be the incidence matrix of the underlying unweighted tree.
	
	Let $\Delta_i$ denote the $i$-th column block of the block matrix $\Delta$. Let $t_i$ be the $n \times 1$ column vector with $1$ at the $i$-th position and $0$ elsewhere, $i=1,2,\hdots,n$. Then
	\begin{equation}\label{eqn1}
		\left[ {\begin{array}{c}
				Q^{\prime}\otimes I_s\\
				t_1^{\prime}\otimes I_s\\
		\end{array} } \right]
		\Delta 
		\left[ {\begin{array}{cc}
				Q\otimes I_s & t_1\otimes I_s\\
		\end{array} } \right]=
		\left[ {\begin{array}{cc}
				(Q^{\prime}\otimes I_s)\Delta (Q\otimes I_s) & (Q^{\prime}\otimes I_s)\Delta_1\\
				\Delta_1^{\prime}(Q\otimes I_s) & 0\\
		\end{array} } \right].
	\end{equation}
	Since $\det\left[ {\begin{array}{c}
			Q^{\prime}\otimes I_s\\
			t_1^{\prime}\otimes I_s\\
	\end{array} } \right]=\det \Bigg( \left[ {\begin{array}{c}
			Q^{\prime}\\
			t_1^{\prime}\\
	\end{array} } \right]\otimes I_s \Bigg)=\pm 1$,  by taking determinant of matrices in both sides of equation (\ref{eqn1}), we have
	\begin{align*}
		\det (\Delta) =&
		\det \left[ {\begin{array}{cc}
				(Q^{\prime}\otimes I_s)\Delta (Q\otimes I_s) & (Q^{\prime}\otimes I_s)\Delta_1\\
				\Delta_1^{\prime}(Q\otimes I_s) & 0\\
		\end{array} } \right].
	\end{align*}
	Using Lemma \ref{lem:FHF}, we have $\det (\Delta)=\det \left[ {\begin{array}{cc}
			-2F(H\otimes I_s)F & (Q^{\prime}\otimes I_s)\Delta_1\\
			\Delta_1^{\prime}(Q\otimes I_s) & 0\\
	\end{array} } \right].$  By Theorem \ref{detH}, we have $\det H=2^{n-2}\prod_{i=1}^n \tau_i$ and hence $\det(H\otimes I_s)=(\det H)^s=2^{(n-2)s}\prod_{i=1}^n \tau_i^s$. Thus, $-2F(H\otimes I_s)F$ is nonsingular, and by the Schur complement formula, we have
	\begin{eqnarray*}
		\det (\Delta) &=& \left[ {\begin{array}{cc}
				-2F(H\otimes I_s)F & (Q^{\prime}\otimes I_s)\Delta_1\\
				\Delta_1^{\prime}(Q\otimes I_s) & 0\\
		\end{array} } \right]\\
		&=& \det(-2F(H\otimes I_s)F)\det \Big(-\Delta_1^{\prime}(Q\otimes I_s)(-2F(H\otimes I_s)F)^{-1}(Q^{\prime}\otimes I_s)\Delta_1\Big)\\
		&=&(-1)^{(n-1)s}2^{(n-2)s}\prod_{i=1}^{n-1}\det(W_i^2) \det(H\otimes I_s)\det\Big(\Delta_1^{\prime}(Q\otimes I_s)F^{-1}(H\otimes I_s)^{-1}F^{-1}(Q^{\prime}\otimes I_s)\Delta_1\Big).
	\end{eqnarray*}
	Now, from Theorem \ref{detH}, it follows that $(H\otimes I_s)^{-1}=H^{-1}\otimes I_s=\frac{1}{2}Q^{\prime}\hat{\tau}Q\otimes I_s=\frac{1}{2}(Q^{\prime}\hat{\tau}Q\otimes I_s)$. Therefore, 
	\begin{equation}\label{eqn det}
		\det (\Delta)=(-1)^{(n-1)s}2^{(2n-5)s}\prod_{i=1}^n {(\tau_i)^s}\prod_{i=1}^{n-1}\det(W_i^2)\det \Big(\Delta_1^{\prime}(Q\otimes I_s)F^{-1}(Q^{\prime}\hat{\tau}Q\otimes I_s)F^{-1}(Q^{\prime}\otimes I_s)\Delta_1\Big). 
	\end{equation}
	Now, by Lemma \ref{deltaL} and Lemma \ref{lem:Ddel}, we have
	\begin{eqnarray*}
		& &\Delta_1^{\prime}(Q\otimes I_s)F^{-1}(Q^{\prime}\hat{\tau}Q\otimes I_s)F^{-1}(Q^{\prime}\otimes I_s)\Delta_1\\
		&=&\Delta_1^{\prime}(Q\otimes I_s)F^{-1}(Q^{\prime}\otimes I_s)(\hat{\tau}\otimes I_s)(Q\otimes I_s)F^{-1}(Q^{\prime}\otimes I_s)\Delta_1\\
		&=&\Big(\Delta_1^{\prime}(Q\otimes I_s)F^{-1}(Q^{\prime}\otimes I_s)\Big)(\hat{\tau}\otimes I_s)\Big(\Delta_1^{\prime}(Q\otimes I_s)F^{-1}(Q^{\prime}\otimes I_s)\Big)^{\prime}\\
		&=&\big(\Delta_1^{\prime}L\big)(\hat{\tau}\otimes I_s)\big(\Delta_1^{\prime}L\big)^{\prime}\\
		&=&\sum_i\big(2\tau_i d_{1i}-\hat{\delta_i}\big)^2\frac{1}{\tau_i}\\
		&=&\sum_i\big(4{\tau_i}^2 d_{1i}^2+{\hat{\delta_i}}^2-4\tau_i d_{1i}\hat{\delta_i}\big)\frac{1}{\tau_i}\\
		&=&\sum_i 4{\tau_i}^2 d_{1i}^2+\sum_i \frac{\hat{\delta_i}^2}{\tau_i}-\sum_i 4d_{1i}\hat{\delta_i}\\
		&=&\sum_i \frac{\hat{\delta_i}^2}{\tau_i}.
	\end{eqnarray*}
	Substituting the value of $\Delta_1^{\prime}(Q\otimes I_s)F^{-1}(Q^{\prime}\hat{\tau}Q\otimes I_s)F^{-1}(Q^{\prime}\otimes I_s)\Delta_1$ in (\ref{eqn det}), we get the required result.
\end{proof}

\begin{figure}
	\centering
	\includegraphics[scale= 0.50]{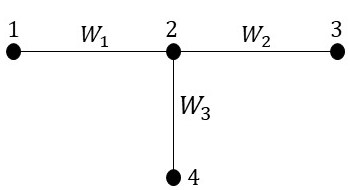}
	\caption{ Tree $T_1$ on 4 vertices}
	\label{fig1}
\end{figure}

Next, let us illustrate the above theorem by an example.
\begin{ex}
	Consider the tree $T_1$ in Figure 1, where the edge weights are 
	\begin{align*}
		W_1=\left[ {\begin{array}{cc}
				1 & 0\\
				0 & 1\\ 
		\end{array} } \right], \qquad
		W_2=\left[ {\begin{array}{cc}
				2 & 0\\
				0 & 1\\ 
		\end{array} } \right], \qquad
		W_3=\left[ {\begin{array}{cc}
				1 & 0\\
				0 & 2\\ 
		\end{array} } \right].   
	\end{align*}
\end{ex}
Then, 
\begin{align*}
	\Delta =&\left[ {\begin{array}{cccc}
			0 & W_1^2 & (W_1+W_2)^2 & (W_1+W_3)^2\\
			W_1^2 & 0 & W_2^2 & W_3^2\\ 
			(W_1+W_2)^2 & W_2^2 & 0 & (W_2+W_3)^2\\ 
			(W_1+W_3)^2 & W_3^2 & (W_2+W_3)^2 & 0\\
	\end{array} } \right]  \\
	=&\left[ {\begin{array}{cccccccc}
			0 & 0 & 1 & 0 & 9 & 0 & 4 & 0\\
			0 & 0 & 0 & 1 & 0 & 4 & 0 & 9\\
			1 & 0 & 0 & 0  & 4 & 0 & 1 & 0\\
			0 & 1 & 0 & 0 & 0 & 1 & 0 & 4\\
			9 & 0 & 4 & 0 & 0 & 0 & 9 & 0\\
			0 & 4 & 0 & 1 & 0 & 0 & 0 & 9\\
			4 & 0 & 1 & 0 & 9 & 0 & 0 & 0 \\
			0 & 9 & 0 & 4 & 0 & 9 & 0 & 0\\
	\end{array} } \right] ~ \text{and}\\   
	\sum_{i=1}^4 \frac{\hat{\delta_i}^2}{\tau_i}=& W_1^2+W_2^2+W_3^2-(W_1+W_2+W_3)^2=
	\left[ {\begin{array}{cc}
			-10 & 0\\
			0 & -10\\ 
	\end{array} } \right]. 
\end{align*}
One can verify that,
$$\det (\Delta)= 102400= (-1)^{6}2^{6}\prod_{i=1}^3 {(\tau_i)^2}\prod_{i=1}^{3}\det({W_i}^2) \det\Big (\sum_{i=1}^4 \frac{\hat{\delta_i}^2}{\tau_i}\Big ).$$ 

Next, we obtain a formula for the determinant of the squared distance matrix of a tree $T$, which has exactly one vertex of degree $2$.

\begin{thm}\label{det}
	Let $T$ be a tree on $n$ vertices with the edge set $E(T)=\{e_1,e_2,\hdots,e_{n-1}\}$. Let the positive definite  matrices $W_1,W_2,\hdots,W_{n-1}$ of order $s$ be the weights of the edges $e_1,e_2,\hdots,e_{n-1}$, respectively. Let $v$ be the vertex of degree $2$ and $u$ and $w$ be its neighbours in $T$. If $e_i=(u,v)$ and $e_j=(v,w)$, then 
	$$\det (\Delta)=-(1)^{(n-1)s}2^{(2n-5)s}\det(W_i+W_j)^2 \prod_{k=1}^{n-1} \det(W_k^2)\prod_{k\neq v}\tau_k^s.$$
\end{thm}
\begin{proof}
	Let us assign an orientation to each edge of $T$. Without loss of generality, assume that, the edge $e_i$ is directed from $u$ to $v$ and the edge $e_j$ is directed from $v$ to $w$. 
	
	Let $\Delta_i$ denote the $i$-th column block of the block matrix $\Delta$. Let $t_i$ be the $n \times 1$ column vector with $1$ at the $i$-th position and $0$ elsewhere, $i=1,2,\hdots,n$.  Therefore, by using Lemma \ref{lem:FHF}, we have
	\begin{eqnarray*}
		\left[ {\begin{array}{c}
				Q^{\prime}\otimes I_s\\
				t_v^{\prime}\otimes I_s\\
		\end{array} } \right]
		\Delta 
		\left[ {\begin{array}{cc}
				Q\otimes I_s & t_v\otimes I_s\\
		\end{array} } \right] &=&
		\left[ {\begin{array}{cc}
				(Q^{\prime}\otimes I_s)\Delta (Q\otimes I_s) & (Q^{\prime}\otimes I_s)\Delta_v\\
				\Delta_v^{\prime}(Q\otimes I_s) & 0\\
		\end{array} } \right]\\
		&=& \left[ {\begin{array}{cc}
				-2F(H\otimes I_s)F) & (Q^{\prime}\otimes I_s)\Delta_v\\
				\Delta_v^{\prime}(Q\otimes I_s) & 0\\
		\end{array} } \right]
	\end{eqnarray*}
	Pre-multiplying and post-multiplying the above equation by $\left[ {\begin{array}{cc}
			F^{-1}& 0\\
			0 & I_s\\
	\end{array} } \right]$, we get
	\begin{eqnarray*}
		\left[ {\begin{array}{cc}
				F^{-1}& 0\\
				0 & I_s\\
		\end{array} } \right]
		\left[ {\begin{array}{c}
				Q^{\prime}\otimes I_s\\
				t_v^{\prime}\otimes I_s\\
		\end{array} } \right]
		\Delta 
		\left[ {\begin{array}{cc}
				Q\otimes I_s & t_v\otimes I_s\\
		\end{array} } \right]
		\left[ {\begin{array}{cc}
				F^{-1}& 0\\
				0 & I_s\\
		\end{array} } \right] &=&
		\left[ {\begin{array}{cc}
				-2(H\otimes I_s) & F^{-1}(Q^{\prime}\otimes I_s)\Delta_v\\
				\Delta_v^{\prime}(Q\otimes I_s)F^{-1} & 0\\
		\end{array} } \right],  
	\end{eqnarray*}
	which implies that
	\begin{eqnarray*}
		(\det(F^{-1}))^2 \det(\Delta) =\det 
		\left[ {\begin{array}{cc}
				-2(H\otimes I_s) & F^{-1}(Q^{\prime}\otimes I_s)\Delta_v\\
				\Delta_v^{\prime}(Q\otimes I_s)F^{-1} & 0\\
		\end{array} } \right].  
	\end{eqnarray*}
	Let $H(j|j)$ denote the $(n-2)s\times (n-2)s$ submatrix obtained by deleting the all blocks in the $j$-th row and $j$-th column from $H\otimes I_s$. Let   $R_i$ and $C_i$ denote the $i$-th row and $i$-th column of the matrix $\left[ {\begin{array}{cc}
			-2(H\otimes I_s) & F^{-1}(Q^{\prime}\otimes I_s)\Delta_v\\
			\Delta_v^{\prime}(Q\otimes I_s)F^{-1} & 0\\
	\end{array} } \right]$, respectively. Note that the blocks in the $i$-th and $j$-th column of $H\otimes I_s$ are identical. Now, perform the operations $R_j-R_i$ and $C_j-C_i$ in $\left[ {\begin{array}{cc}
			-2(H\otimes I_s) & F^{-1}(Q^{\prime}\otimes I_s)\Delta_v\\
			\Delta_v^{\prime}(Q\otimes I_s)F^{-1} & 0\\
	\end{array} } \right]$, and then interchange $R_j$ and $R_{n-1}$, $C_j$ and $C_{n-1}$ . Since $\Delta_v^{\prime}(Q\otimes I_s)F^{-1})_j-( \Delta_v^{\prime}(Q\otimes I_s)F^{-1})_i=-W_j-W_i$, therefore 
	\begin{equation}
		\det \left[ {\begin{array}{cc}
				-2(H\otimes I_s) & F^{-1}(Q^{\prime}\otimes I_s)\Delta_v\\
				\Delta_v^{\prime}(Q\otimes I_s)F^{-1} & 0\\
		\end{array} } \right] = \det \left[ {\begin{array}{ccc}
				-2H(j|j) & 0 & F^{-1}(Q^{\prime}\otimes I_s)\Delta_v\\
				0 & 0 & -W_j-W_i\\
				\Delta_v^{\prime}(Q\otimes I_s)F^{-1} & -W_j-W_i & 0\\
		\end{array} } \right]. 
	\end{equation}
	Since $H(j|j)$ is the edge orientation matrix of the tree obtained by deleting the vertex $v$ and replacing the edges $e_i$ and $e_j$ by a single edge directed from $u$ to $w$ in the tree, by Theorem \ref{detH}, we have
	$\det(H(j|j)=2^{(n-3)s}\prod_{k \neq v}\tau_k^s$, which is nonzero. Therefore, by applying the Schur complement formula, we have 
	\begin{eqnarray*}
		& &\det \left[ {\begin{array}{ccc}
				-2H(j|j) & 0 & F^{-1}(Q^{\prime}\otimes I_s)\Delta_v\\
				0 & 0 & -W_j-W_i\\
				\Delta_v^{\prime}(Q\otimes I_s)F^{-1} & -W_j-W_i & 0\\
		\end{array} } \right] \\
		&=& \det(-2H(j|j)) \det \bigg(\left[ {\begin{array}{cc}
				0 & -W_j-W_i\\
				-W_j-W_i & 0\\
		\end{array} } \right]-\\ & &~~~~~~~~~~~~~~~~~~~~~~~~~~~
		\left[ {\begin{array}{cc}
				0 & 0 \\
				0 & \Delta_v^{\prime}(Q\otimes I_s)F^{-1}(-2H(j|j))^{-1}F^{-1}(Q^{\prime}\otimes I_s)\Delta_v\\
		\end{array} } \right]  \bigg)\\
		&=&(-2)^{(n-2)s}\det(H(j|j)) \det \left[ {\begin{array}{cc}
				0 & -W_j-W_i\\
				-W_j-W_i & -\Delta_v^{\prime}(Q\otimes I_s)F^{-1}(-2H(j|j))^{-1}F^{-1}(Q^{\prime}\otimes I_s)\Delta_v\\
		\end{array} } \right].
	\end{eqnarray*}
	Again, by the proof of Theorem \ref{det1}, we have $$\Delta_v^{\prime}(Q\otimes I_s)F^{-1}(-2H(j|j))^{-1}F^{-1}(Q^{\prime}\otimes I_s)\Delta_v=-\frac{1}{4}\sum_{i\neq v} \frac{\hat{\delta_i}^2}{\tau_i}.$$ Therefore,  
	\begin{eqnarray*}
		& &\det \left[ {\begin{array}{ccc}
				-2H(j|j) & 0 & F^{-1}(Q^{\prime}\otimes I_s)\Delta_v\\
				0 & 0 & -W_j-W_i\\
				\Delta_v^{\prime}(Q\otimes I_s)F^{-1} & -W_j-W_i & 0\\
		\end{array} } \right] \\
		&=& (-2)^{(n-2)s}\det(H(j|j)) \det \left[ {\begin{array}{cc}
				0 & -W_j-W_i\\
				-W_j-W_i & \frac{1}{4}\sum_{i\neq v} \frac{\hat{\delta_i}^2}{\tau_i}\\
		\end{array} } \right]\\ 
		&=& (-2)^{(n-2)s}\det(H(j|j)) \det \left[ {\begin{array}{cc}
				0 & W_j+W_i\\
				W_j+W_i & -\frac{1}{4}\sum_{i\neq v} \frac{\hat{\delta_i}^2}{\tau_i}\\
		\end{array} } \right].
	\end{eqnarray*}
	Since $\det \Big(-\frac{1}{4}\sum_{i\neq v} \frac{\hat{\delta_i}^2}{\tau_i}\Big)\neq 0$, by Schur complement formula, we have
	\begin{eqnarray*}
		\det \left[ {\begin{array}{cc}
				0 & W_j+W_i\\
				W_j+W_i & -\frac{1}{4}\sum_{i\neq v} \frac{\hat{\delta_i}^2}{\tau_i}\\
		\end{array} } \right]
		&=&\det \bigg(-\frac{1}{4}\sum_{i\neq v} \frac{\hat{\delta_i}^2}{\tau_i}\bigg) \det \bigg[0-(W_j+W_i) \bigg(-\frac{1}{4}\sum_{i\neq v} \frac{\hat{\delta_i}^2}{\tau_i}\bigg)^{-1}( W_j+W_i)\bigg]\\
		&=&(-1)^s \det \bigg(-\frac{1}{4}\sum_{i\neq v} \frac{\hat{\delta_i}^2}{\tau_i}\bigg) \det \bigg(-\frac{1}{4}\sum_{i\neq v} \frac{\hat{\delta_i}^2}{\tau_i}\bigg)^{-1} \det(W_j+W_i)^2\\
		&=&(-1)^s \det(W_i+W_j)^2.
	\end{eqnarray*}
	
	Thus,
	\begin{eqnarray*}
		\det (\Delta) &=&(\det F)^2(-1)^{s}(-2)^{(n-2)s}2^{(n-3)s}\prod_{k\neq v}\tau_k^s~\det(W_i+W_j)^2\\
		&=&(-1)^{(n-1)s}2^{(2n-5)s}\det(W_i+W_j)^2\prod_{k=1}^{n-1}\det(W_k^2)\prod_{k\neq v}\tau_k^s.
	\end{eqnarray*}
	This completes the proof.
\end{proof}

\begin{figure}
	\centering
	\includegraphics[scale= 0.50]{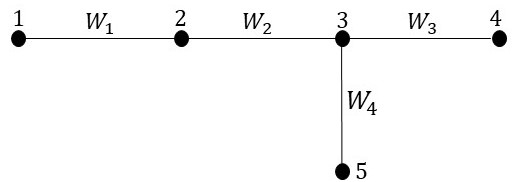}
	\caption{ Tree $T_2$ on 5 vertices }
	\label{fig2}
\end{figure}

Now, we illustrate the above theorem by the following example.
\begin{ex}
	Consider the tree $T_2$ in Figure \ref{fig2}, where the edge weights are 
	\begin{align*}
		W_1=\left[ {\begin{array}{cc}
				1 & 0\\
				0 & 1\\ 
		\end{array} } \right], \qquad
		W_2=\left[ {\begin{array}{cc}
				2 & 0\\
				0 & 1\\ 
		\end{array} } \right], \qquad
		W_3=\left[ {\begin{array}{cc}
				1 & 0\\
				0 & 2\\ 
		\end{array} } \right], \qquad
		W_4=\left[ {\begin{array}{cc}
				2 & 0\\
				0 & 2\\ 
		\end{array} } \right].   
	\end{align*}
\end{ex}
Then, 
\begin{eqnarray*}
	\Delta &=&\left[ {\begin{array}{ccccc}
			0 & W_1^2 & (W_1+W_2)^2 & (W_1+W_2+W_3)^2 & (W_1+W_2+W_4)^2\\
			W_1^2 & 0 & W_2^2 & (W_2+W_3)^2 & (W_2+W_4)^2\\ 
			(W_1+W_2)^2 & W_2^2 & 0 & W_3^2 & W_4^2\\ 
			(W_1+W_2+W_3)^2 &(W_2+W_3)^2 & W_3^2 & 0 & (W_3+W_4)^2\\
			(W_1+W_2+W_3)^2 & (W_2+W_4)^2 & W_4^2 & (W_3+W_4)^2 & 0\\
	\end{array} } \right]  \\
	&=&\left[ {\begin{array}{cccccccccc}
			0 & 0 & 1 & 0 & 9 & 0 & 16 & 0 & 25 & 0\\
			0 & 0 & 0 & 1 & 0 & 4 & 0 & 16 & 0 & 16\\
			1 & 0 & 0 & 0  & 4 & 0 & 9 & 0 & 16 & 0\\
			0 & 1 & 0 & 0 & 0 & 1 & 0 & 9 & 0 & 9\\
			9 & 0 & 4 & 0 & 0 & 0 & 1 & 0 & 4 & 0\\
			0 & 4 & 0 & 1 & 0 & 0 & 0 & 4 & 0 & 4\\
			16 & 0 & 9 & 0 & 1 & 0 & 0 & 0  & 9 & 0\\
			0 & 16 & 0 & 9 & 0 & 4 & 0 & 0 & 0 & 16\\
			25 & 0 & 16 & 0 & 4 & 0 & 9 & 0 & 0 & 0 \\
			0 & 16 & 0 & 9 & 0 & 4 & 0 & 16 & 0 & 0 \\
	\end{array} } \right].    
\end{eqnarray*}
One can verify that, $$\det (\Delta)= 9437184= (-1)^{8}2^{10}\det(W_1+W_2)^2 \prod_{i=1}^{4} \det(W_i^2)\prod_{k\neq 2}\tau_k^s.$$

\begin{cor}
	Let $T$ be a tree on $n$ vertices and each edge $e_i$ of $T$ is assigned a  positive definite matrix $W_i$ order $s$, $i=1,2,\hdots,n-1$. If $T$ has at least two vertices of degree $2$, then $\det (\Delta)=0$.
\end{cor}
\begin{proof}
	The result follows from Theorem \ref{det}, since $\tau_i=0$ for at least two values of $i$.
\end{proof} 

\section{Inverse of the squared distance matrix}
This section considers trees with no vertex of degree $2$ and obtains an explicit formula for the inverse of its squared distance matrix. First, let us prove the following lemma which will be used to find $\Delta^{-1}$.

\begin{lem}\label{lem:inv}
	Let $T$ be a tree of order $n$ with no vertex of degree $2$ and each edge of $T$ is assigned a positive definite matrix weight of order $s$. If $\beta=\Hat{{\delta}^{\prime}}(\Hat{\tau}\otimes I_s)\Hat{\delta}$ and $\eta=2\tau \otimes I_s-L(\hat{\tau}\otimes I_s)\Hat{\delta}$, then 
	$$\Delta \eta =\textbf{1}_n \otimes \beta.$$
\end{lem}
\begin{proof}
	By Lemma \ref{deltaL}, we have $\Delta L=2D(\Tilde{\tau}\otimes I_s)-\textbf{1}_n \otimes {\hat{\delta}^\prime}$. Hence, 
	\begin{eqnarray*}
		\Delta L(\Hat{\tau}\otimes I_s)\hat{\delta}&=&2D\hat{\delta}-(\textbf{1}_n \otimes {\hat{\delta}^\prime})(\Hat{\tau}\otimes I_s)\hat{\delta}\\
		&=&2D\hat{\delta}-\textbf{1}_n \otimes
		\sum_{i=1}^n\frac{\hat{\delta_i}^2}{\tau_i}. 
	\end{eqnarray*}
	Since $\beta=\Hat{{\delta}^{\prime}}(\Hat{\tau}\otimes I_s)\Hat{\delta}=\sum_{i=1}^n\frac{\hat{\delta_i}^2}{\tau_i}$, therefore
	$\Delta L(\Hat{\tau}\otimes I_s)\hat{\delta}=2D\hat{\delta}-\textbf{1}_n \otimes \beta$. By Lemma \ref{lem:Ddel}, we have $\Delta (\tau \otimes I_s) =D \hat{\delta}$ and hence $\Delta L(\Hat{\tau}\otimes I_s)\hat{\delta}= 2\Delta (\tau \otimes I_s)-\textbf{1}_n\otimes \beta$. This completes the proof.
\end{proof}

If the tree $T$ has no vertex of degree $2$ and $\det(\beta) \neq 0$, then $\Delta$ is nonsingular, that is, ${\Delta}^{-1}$ exists. In the next theorem, we determine the formula for ${\Delta}^{-1}$.

\begin{thm}
	Let $T$ be a tree of order $n$ with no vertex of degree $2$ and each edge of $T$ is assigned a positive definite matrix weight of order $s$. Let $\beta=\Hat{{\delta}^{\prime}}(\Hat{\tau}\otimes I_s)\Hat{\delta}$ and $\eta=2\tau \otimes I_s-L(\hat{\tau}\otimes I_s)\Hat{\delta}$. If $\det(\beta) \neq 0$, then 
	$${\Delta}^{-1}=-\frac{1}{4}L(\Hat{\tau}\otimes I_s)L+\frac{1}{4}\eta {\beta}^{-1} {\eta}^{\prime}.$$
\end{thm}
\begin{proof}
	Let $X=-\frac{1}{4}L(\Hat{\tau}\otimes I_s)L+\frac{1}{4}\eta {\beta}^{-1} {\eta}^{\prime}$.
	Then, 
	\begin{equation}\label{eqn:inv1}
		\Delta X=-\frac{1}{4}\Delta L(\Hat{\tau}\otimes I_s)L+\frac{1}{4}\Delta \eta {\beta}^{-1} {\eta}^{\prime}.
	\end{equation}
	By Lemma \ref{deltaL}, we have $\Delta L=2D(\Tilde{\tau}\otimes I_s)-\textbf{1}_n\otimes {\hat{\delta}^\prime}$. Therefore, 
	$$\Delta L(\Hat{\tau}\otimes I_s)L=2DL-(\textbf{1}_n\otimes {\hat{\delta}^\prime})(\Hat{\tau}\otimes I_s)L. $$
	By Theorem \ref{thm:DL}, $DL=\textbf{1}_n{\tau}^{\prime}\otimes I_s-2I_n\otimes I_s$ and hence
	\begin{equation}\label{eqn:inv2}
		\Delta L(\Hat{\tau}\otimes I_s)L=2\Big(\textbf{1}_n{\tau}^{\prime}\otimes I_s-2I_n\otimes I_s\Big)-(\textbf{1}_n\otimes {\hat{\delta}^\prime})(\Hat{\tau}\otimes I_s)L.
	\end{equation}
	By Lemma \ref{lem:inv}, we have $\Delta \eta =\textbf{1}_n\otimes \beta=(\textbf{1}_n\otimes I_s)\beta$. Therefore, from equation (\ref{eqn:inv1}) and (\ref{eqn:inv2}), we have 
	\begin{eqnarray*}
		\Delta X &=& -\frac{1}{2}\Big(\textbf{1}_n{\tau}^{\prime}\otimes I_s-2I_n\otimes I_s\Big)+\frac{1}{4}(\textbf{1}_n\otimes {\hat{\delta}^\prime})(\Hat{\tau}\otimes I_s)L+\frac{1}{4}(\textbf{1}_n \otimes I_s){\eta}^{\prime}\\
		& = & -\frac{1}{2}\textbf{1}_n{\tau}^{\prime}\otimes I_s+I_n\otimes I_s+\frac{1}{4}(\textbf{1}_n\otimes {\hat{\delta}^\prime})(\Hat{\tau}\otimes I_s)L+\frac{1}{4}(\textbf{1}_n\otimes I_s)\Big(2\tau \otimes I_s-L(\hat{\tau}\otimes I_s)\Hat{\delta}\Big)^{\prime}\\
		& = & -\frac{1}{2}\textbf{1}_n{\tau}^{\prime}\otimes I_s+I_n\otimes I_s+\frac{1}{4}(\textbf{1}_n\otimes {\hat{\delta}^\prime})(\Hat{\tau}\otimes I_s)L+\frac{1}{4}(\textbf{1}_n\otimes I_s)\Big(2\tau^{\prime} \otimes I_s-{\Hat{\delta}}^{\prime}(\hat{\tau}\otimes I_s)L\Big)\\
		&=& I_n\otimes I_s=I_{ns}.
	\end{eqnarray*}
	This completes the proof.
\end{proof}
Now, let us illustrate the above formula for $\Delta^{-1}$ by an example.
\begin{ex}
	Consider the tree $T_1$ in Figure 1, where the edge weights are 
	\begin{align*}
		W_1=\left[ {\begin{array}{cc}
				1 & 0\\
				0 & 1\\ 
		\end{array} } \right], \qquad
		W_2=\left[ {\begin{array}{cc}
				2 & 0\\
				0 & 1\\ 
		\end{array} } \right], \qquad
		W_3=\left[ {\begin{array}{cc}
				1 & 0\\
				0 & 2\\ 
		\end{array} } \right].   
	\end{align*}
\end{ex}
Then, 
\begin{align*}
	\Delta =&\left[ {\begin{array}{cccc}
			0 & W_1^2 & (W_1+W_2)^2 & (W_1+W_3)^2\\
			W_1^2 & 0 & W_2^2 & W_3^2\\ 
			(W_1+W_2)^2 & W_2^2 & 0 & (W_2+W_3)^2\\ 
			(W_1+W_3)^2 & W_3^2  & (W_2+W_3)^2 & 0\\
	\end{array} } \right]  \\
	=&\left[ {\begin{array}{cccccccc}
			0 & 0 & 1 & 0 & 9 & 0 & 4 & 0\\
			0 & 0 & 0 & 1 & 0 & 4 & 0 & 9\\
			1 & 0 & 0 & 0  & 4 & 0 & 1 & 0\\
			0 & 1 & 0 & 0 & 0 & 1 & 0 & 4\\
			9 & 0 & 4 & 0 & 0 & 0 & 9 & 0\\
			0 & 4 & 0 & 1 & 0 & 0 & 0 & 9\\
			4 & 0 & 1 & 0 & 9 & 0 & 0 & 0 \\
			0 & 9 & 0 & 4 & 0 & 9 & 0 & 0\\
	\end{array} } \right],\\
	L=&\left[ {\begin{array}{cccc}
			W_1^{-1}& -W_1^{-1} & 0 & 0\\
			-W_1^{-1} & W_1^{-1}+W_2^{-1}+W_3^{-1} & -W_2^{-1} & -W_3^{-1}\\ 
			0 & -W_2^{-1} & W_2^{-1} & 0 \\ 
			0 & -W_3^{-1} & 0  &W_3^{-1}\\
	\end{array} } \right]  \\
	=&\left[ {\begin{array}{cccccccc}
			1 & 0 & -1 & 0 & 0 & 0 & 0 & 0\\
			0 & 1 & 0 & -1 & 0 & 0 & 0 & 0\\
			-1 & 0 & 2.5 & 0  & -0.5 & 0 & -1 & 0\\
			0 & -1 & 0 & 2.5 & 0 & -1 & 0 & -0.5\\
			0 & 0 & -0.5 & 0 & 0.5 & 0 & 0 & 0\\
			0 & 0 & 0 & -1 & 0 & 1 & 0 & 0\\
			0 & 0 & -1 & 0 & 0 & 0 & 1 & 0 \\
			0 & 0 & 0 & -0.5 & 0 & 0 & 0 & 0.5\\
	\end{array} } \right],  
\end{align*}
\begin{align*}
	\beta =\sum_{i=1}^4 \frac{\hat{\delta_i}^2}{\tau_i}=& W_1^2+W_2^2+W_3^2-(W_1+W_2+W_3)^2=\left[ {\begin{array}{cc}
			-10 & 0\\
			0 & -10\\ 
	\end{array} } \right], ~ \text{and}\\
	{\eta}^{\prime} =& \left[ {\begin{array}{cccccccc}
			-3 & 0 & 11 & 0 & -1 & 0 & -3 & 0\\
			0 & -3 & 0 & 11 & 0 & -3 & 0 & -1\\ 
	\end{array} } \right].
\end{align*}
Therefore,
\begin{align*}
	L(\Hat{\tau}\otimes I_s)L=& \left[ {\begin{array}{cccccccc}
			0 & 0 & 1.5 & 0 & -0.5 & 0 & -1 & 0\\
			0 & 0 & 0 & 1.5 & 0 & -1 & 0 & -0.5\\
			1.5 & 0 & -4& 0  & 1 & 0 & 1.5 & 0\\
			0 & 1.5 & 0 & -4 & 0 & 1.5 & 0 & 1\\
			-0.5 & 0 & 1 & 0 & 0 & 0 & -0.5 & 0\\
			0 & -1 & 0 & 1.5 & 0 & 0 & 0 & -0.5\\
			-1 & 0 & 1.5 & 0 & -0.5 & 0 & 0 & 0 \\
			0 & -0.5 & 0 & 1 & 0 & -0.5 & 0 & 0\\
	\end{array} } \right],~ \text{and}  \\
	\eta {\beta}^{-1} {\eta}^{\prime}=& \left[ {\begin{array}{cccccccc}
			-0.9 & 0 & 3.3 & 0 & -0.3 & 0 & -0.9 & 0\\
			0 & -0.9 & 0 & 3.3 & 0 & -0.9 & 0 & -0.3\\
			3.3 & 0 & -12.1 & 0  & 1.1 & 0 & 3.3 & 0\\
			0 & 3.3 & 0 & -12.1 & 0 & 3.3 & 0 & 1.1\\
			-0.3 & 0 & 1.1 & 0 & -0.1 & 0 & -0.3 & 0\\
			0 & -0.9 & 0 & 3.3 & 0 & -0.9 & 0 & -0.3\\
			-0.9 & 0 & 3.3 & 0 & -0.3 & 0 & -0.9 & 0 \\
			0 & -0.3 & 0 & 1.1 & 0 & -0.3 & 0 & -0.1\\
	\end{array} } \right].
\end{align*}
One can verify that, $$\Delta^{-1}=-\frac{1}{4}L(\Hat{\tau}\otimes I_s)L+\frac{1}{4}\eta {\beta}^{-1} {\eta}^{\prime}.$$

\bibliographystyle{plain}
\bibliography{reference}

\end{document}